\documentclass[reqno]{amsart}

\usepackage{url}
\usepackage{a4wide}

\usepackage[utf8]{inputenc}
\usepackage[T1]{fontenc}
\usepackage{lmodern}

\usepackage{amsmath}
\usepackage{amssymb}
\usepackage{amsfonts}
\usepackage{amscd}
\usepackage{amsthm}
\usepackage{caption}
\usepackage{subcaption}

\usepackage{graphicx}
\usepackage[english]{babel}

\usepackage{wrapfig}
\usepackage{breqn}

\usepackage[all]{xy}
\usepackage{hyperref}

\usepackage[ruled,vlined,boxed,linesnumbered,norelsize]{algorithm2e}

\allowdisplaybreaks

\theoremstyle{plain}
\newtheorem{theorem}{Theorem}[section]

\newtheorem{proposition}[theorem]{Proposition}
\newtheorem{lemma}[theorem]{Lemma}
\newtheorem{corollary}[theorem]{Corollary}

\theoremstyle{definition}
\newtheorem{definition}[theorem]{Definition}

\theoremstyle{remark}
\newtheorem{remark}[theorem]{Remark}

\numberwithin{equation}{section}

\DeclareMathOperator{\Aut}{Aut}

\DeclareMathOperator{\GL}{GL}

\DeclareMathOperator{\PGL}{PGL}

\DeclareMathOperator{\Prob}{Prob}

\def\F{\mathbf{F}}

\def\ie{\textit{i.e.}}
\def\cf{\textit{c.f.}}

\newcommand{\Mm}{\mathsf{M}}

\newcommand{\AAb}{\mathbb{A}}
\newcommand{\FF}{\mathbb{F}}

\newcommand{\PP}{\mathbb{P}}

\newcommand{\ZZ}{\mathbb{Z}}

\newcommand{\nn}{\mathcal{N}}

\newcommand{\vv}{\mathcal{V}}

\usepackage{mdwlist}

\usepackage{color}

\def\phi{\varphi}
\def\Pt{\tilde{P}}
\def\et{\text{ét}}

\hypersetup{ 
  pdfauthor   = {Reynald Lercier, Christophe Ritzenthaler, Florent Rovetta, Jeroen Sijsling, Ben Smith}
  pdftitle    = {Distributions of traces of Frobenius}
  pdfsubject  = {},
  pdfkeywords = {},
  backref=true, pagebackref=true, hyperindex=true, colorlinks=true,
  breaklinks=true, urlcolor=blue, linkcolor=blue, citecolor=blue, bookmarks=true,
  bookmarksopen=true}

\title[Distributions of traces of Frobenius]{Distributions of traces of
  Frobenius \\ for smooth plane curves over finite fields}

\author[Lercier]{Reynald Lercier}
\address{%
  \textsc{DGA MI}, %
  La Roche Marguerite, %
  35174 Bruz, %
  France. %
}
\address{%
  Institut de recherche math\'ematique de Rennes, %
  Universit\'e de Rennes 1, %
  Campus de Beaulieu, %
  35042 Rennes, %
  France. %
} 
\email{reynald.lercier@m4x.org}

\author[Ritzenthaler]{Christophe Ritzenthaler}
\address{%
  Institut de recherche math\'ematique de Rennes, %
  Universit\'e de Rennes 1, %
  Campus de Beaulieu, %
  35042 Rennes, %
  France. %
}
\email{christophe.ritzenthaler@univ-rennes1.fr }

\author[Rovetta]{Florent Rovetta}
\address{%
  Institut de Math{\'e}matiques de Luminy, %
  UMR 6206 du CNRS, %
  Luminy, Case 907, 13288 Marseille, %
  France. %
}
\email{florent.rovetta@univ-amu.fr}

\author[Sijsling]{Jeroen Sijsling} 
\address{%
  Department of Mathematics, %
  27 N.\ Main Street, %
  6188 Kemeny Hall, %
  Hanover, NH 03755-3551, %
  United States of America. %
}
\email{sijsling@gmail.com}

\author[Smith]{Benjamin Smith}
\address{%
  INRIA, %
  Laboratoire d'Informatique de l'\'Ecole polytechnique (LIX), %
  91128 Palaiseau, %
  France. %
}
\email{smith@lix.polytechnique.fr}

\thanks{The authors acknowledge support by grant ANR-09-BLAN-0020-01.}

\date{\today} 

\begin{document}

\maketitle

\begin{abstract}
  In a previous article, we obtained data on the distribution of traces of
  Frobenius of non-hyperelliptic genus $3$ curves over small finite fields. In
  the present one, we give a heuristic explanation of these data, by
  extrapolating from results on the distribution of traces of Frobenius for
  plane curves whose degree is small with respect to the cardinality of their
  finite base field. In particular, our methods shed some new light on the
  asymmetry of the distribution around its mean value, which is related to the
  Serre obstruction.
\end{abstract}

\section{Introduction}\label{sec:introduction}

More than 30 years ago, Serre found closed formul\ae{} for the maximal possible
number of rational points on a curve of genus \(g \le 2\) over a finite field
$\FF_q$. The same article~\cite{serre-point} also considered the problem of
obtaining a similar formula for curves of genus $3$. This problem is more
involved, as Serre indeed noted at the time. The main obstruction to obtaining a
closed formula is now known as the \emph{Serre obstruction}: a principally
polarized abelian threefold over~$\FF_q$ that becomes the Jacobian of a
genus~$3$ curve after base extension to~$\overline{\FF}_q$ need not be the
Jacobian of such a curve over~$\FF_q$. There is no such obstruction in dimension
1, because elliptic curves are their own Jacobians; nor in dimension 2, because
all genus $2$ curves are hyperelliptic (and thus their automorphism groups are
isomorphic to those of their Jacobians, which which implies in general that such
a descent is possible). The Serre obstruction appears in dimension~$3$ precisely
because of the existence of non-hyperelliptic curves of genus~$3$.

There have been many computational and conceptual approaches to the Serre
obstruction. Partial results can be found in~\cite{ibug3, lauterg3, mestreg3,
NR08, NR09, Rit-serre}, while data for small fields can be found at
\url{manypoints.org}. More general approaches were considered in~\cite{BeauRit,
LR, LRZ}. However, so far none of these results can be used to give a closed
formula in the sense of~\cite{serre-point}.

More precisely, the issue is as follows. The number of rational points on a
given curve $C$ of genus $g$ over a finite field $\FF_q$ is given by
\[
  \# C(\FF_q) = q + 1 - t\ .
\]
Here $t$ is the trace of Frobenius acting on the first \'etale cohomology group
$H^1 (J_{\et}, \ZZ_{\ell})$ of the Jacobian $J$ of $C$, with $\ell$ some prime
not equal to $p$. The classical Hasse--Weil--Serre bound shows that the absolute
value of $t$ is bounded by $2 g \sqrt{q}$. The usual strategy for constructing
maximal curves is to construct a principally polarized abelian variety $A$ over
$\FF_q$ with a trace $t$ close to $-2 g \sqrt{q}$, so as to obtain a number of
rational points that is close to the upper bound $q + 1 + 2g \sqrt{q}$.


The Serre obstruction amounts to the fact that $A$ is not necessarily the
Jacobian of a curve $C$ over $\FF_q$. If it is not, then the quadratic twist of
$A$ will be a Jacobian instead. However, taking such a quadratic twist changes
the trace of Frobenius $t$ to $-t$, so that the resulting $C$ has few points
instead of many points. 


This article was motivated by the hope of finding a new inroad to the problem,
not by studying all descent problems for individual abelian varieties, but
rather by an indirect approach: namely, by proving that for large negative
values of $t$ there are more curves with trace $t$ than with trace~$-t$. More
precisely, let $\nn_{q,3}(t)$ be the number of non-hyperelliptic genus $3$
curves over $\FF_q$ of trace $t$ up to $\FF_q$-isomorphism. We are interested in
studying the difference
\[
  \vv_{q,3} (t) := \nn_{q,3} (t) - \nn_{q,3} (-t)
\]
for $0 \leq t \leq 6\sqrt{q}$, and more specifically in proving that $\vv_{q,3}
(t) \leq 0$  for large enough $t$. This would be enough to show (using
\cite{lauterg3}) that there always exists a curve $C$ such that $\#C(\FF_q) \geq
q + 1 + 3 \lfloor 2 \sqrt{q} \rfloor -3$, and would moreover allow us to give
the precise maximal value (see~\cite[Prop.4.1.7]{HDRritz}).

A numerical study of $\vv_{q,3}(t)$ for prime fields $\FF_q$ with $11 \leq q
\leq 53$ appears in~\cite{LRRS14}. Over these small fields, it was possible to
construct all smooth non-hyperelliptic curves of genus~$3$ (that is, all smooth
plane quartics) up to $\FF_q$-isomorphism, and to compute the trace of a
representative of each $\FF_q$-isomorphism class. The resulting functions
$\vv_{q,3} (t)$ had some remarkable properties: for example, it always appeared
that $\vv_{q,3} (t)$ was negative for $t > 1.7\sqrt{q}$, so that the
corresponding number of curves with many points was larger than the number of
curves with few points. Our original hope that $\vv_{q,3}(t)$ would always be
negative for large enough $t$ turned out to be false in general, as was noted in
\cite{LRRS14}. Nevertheless, for each $q$ the data obtained fitted simple and
pleasing graphs, and moreover these graphs almost coincided after normalizing by
an explicit power of $q$.

This phenomenon seemed interesting enough to merit further investigation.
Moreover, it fits into a larger framework of results on the distribution of the
number of rational points on curves over finite fields, as studied
in~\cite{bucur-tri, bucurp, bucurs, cheongb, entin, kurlberg, wood, xiong,
xiongb} and especially~\cite{bucur}. The results of Bucur--David--Feigon--Lalín
in~\cite{bucur} show that the number of rational points on plane curves of
degree $d$ over $\FF_q$ is distributed according to an explicit and intuitive
binomial law as long as $d$ is large enough compared to $q$. We study this
binomial law, or rather the normalized variation of this law around its mean, in
Section \ref{sec:bino}. By the central limit theorem, this variation converges
to $0$; however, we show that after multiplying by a factor of $\sqrt{q}$, its
behavior as $q$ tends to infinity can be approximated by the function
\[
  \psi : x \mapsto \frac{1}{3 \sqrt{2 \pi}} x (3 - x^2) e^{-x^2 / 2}\  .
\]

In Section~\ref{sec:pointcount}, using~\cite{bucur}, we study the variation
around the mean of the distribution of the number of rational points on plane
curves of degree $d$ over $\FF_q$ as $q \to \infty$ for both general curves
(Corollary~\ref{cor:diffsingular}) and smooth curves
(Corollary~\ref{cor:diffsmooth}). Along the way, we show by elementary methods
that the bound $d \geq q^2 + q$ of~\cite[Prop.~1.6]{bucur} can be improved to $d
\geq 2q - 1$ (in Proposition~\ref{prop:bucurplus}). We note that in these cases,
where $d$ is large compared to $q$, the fact that there are more curves with
large negative trace $t$ than with trace $-t$ is easy to see, since a curve with
a large positive trace would have a negative number of rational points.

In Section 4, in contrast to the approach for sufficiently large $d$ above, we
instead fix the particular small value $d = 4$, while still letting $q$ tend to
$\infty$. Here it remains a challenge to prove any exact results, but the
comparison with our experiments in~\cite{LRRS14} is impressive. It would be
especially interesting to see whether the observed phenomena persist for larger
$q$ than those considered in these experiments. 
Our results are also consistent with those of~\cite{achter-kedlaya}, where it is
conjectured that the distribution of the fraction of curves of genus $g$ over a
fixed finite field $\FF_q$ whose number of points equals $n$ tends to a Poisson
distribution as $g$ tends to infinity. More precisely, it is suggested as one
runs over a set of curves over $\FF_q$ that represent the $\FF_q$-rational
points of the moduli stack of curves $\Mm_g$, one should have the following limit
behavior:
\[
  \lim_{g \to \infty} \# \left\{ C \in \Mm_g ( \F_q ) : \# C (\FF_q) = n \right\}
  = \frac{\lambda^n e^{-\lambda}}{n!}
    \quad \text{where $\lambda = q + 1 + \frac{1}{q-1}$\ .}
\]
The same argument as in the proof of Proposition~\ref{pro:lb} shows that the
variation around the mean gives rise to the same distribution that we obtain in
our paper.

\subsection*{Acknowledgments} 
We are very grateful to Mohamed Barakat for his helpful comments,
Masaaki Homma for his proof of Lemma~\ref{lemma:freeres}, Everett Howe for email
exchanges which led us to the heuristic interpretation described in this
article, and Atilla Yilmaz for pointing out the link between our statistical
result and Edgeworth series.


\section{Some remarks on the binomial distribution}
\label{sec:bino}

We begin by analyzing the asymmetry  of the binomial distribution around its
mean in a general context. We will then adapt the parameters according to our
arithmetical problems.

%
Let us consider the the binomial random variable of standard deviation $\sigma$
\[
  S_{\sigma} := \sum_{i=1}^{N_{\sigma}} B_{\sigma,i}
  \ 
\]
given by $N_{\sigma}$ Bernoulli random variables $B_{\sigma,i}$ taking the value
$1$ with probability $\mu_{\sigma}$. Let 
\[
  b_{\sigma} (m) 
  :=  
  \Prob(S_{\sigma} = m) 
  = 
  {N_{\sigma} \choose m} \mu_{\sigma}^m (1-\mu_{\sigma})^{(N_{\sigma} - m)}
\]
be the corresponding binomial mass function of mean 
$E_{\sigma} = N_{\sigma} \mu_{\sigma}$.
We have the usual relation 
\[
  \sigma = \sqrt{N_{\sigma} \mu_{\sigma} (1 - \mu_{\sigma})} \ .
\]

Let us now consider an increasing sequence of $\sigma$. For any fixed real
$\alpha > 0$, we define a sequence of intervals $I_{\sigma} =
[-\sigma^{1-\alpha},\sigma^{1-\alpha}]$. In order to deal with approximations in
both $\sigma$ and $x \in I_{\sigma}$, we need one more piece of notation.

\begin{definition}\label{def:ofamily}
  Suppose that for all (sufficiently large)  ${\sigma}$ we are given an interval
  $I_{\sigma}$ along with functions $f_{\sigma}$ and $g_{\sigma}$ on
  $I_{\sigma}$. We write $f_{\sigma} = o (g_{\sigma})$ if for all $\epsilon >0$,
  there exists a ${\sigma}_0$ such that 
  \[
    {|f_{\sigma}(x)|} \leq \epsilon |g_{\sigma}(x)|
    \quad 
    \forall x \in I_{\sigma}\ ,\ 
    \forall {\sigma} > {\sigma}_0 
    \ .
  \]
  The notation $f_{\sigma} = O (g_{\sigma})$ is defined similarly.
\end{definition}



The triangular central limit theorem~\cite[Th.27.3]{billingsley}
shows\footnote{
  This can also be deduced from the De Moivre--Laplace theorem~\cite[Ex.
  25.11]{billingsley}, which claims that if $(m_{\sigma})_{\sigma}$ is any
  sequence of integers such that $(m_{\sigma} - E_\sigma)/\sigma$ tends to some
  real $x$ as ${\sigma} \to \infty$ then
  \( 
    \lim_{{\sigma} \to \infty} \sigma \Prob(S_{{\sigma}} = m_{\sigma}) = \frac{1}{\sqrt{2 \pi}} e^{-\frac{1}{2}
  x^2}  
  \).
  }
that the normalized average sequence
\(
  Y_{{\sigma}} := (S_{\sigma} - E_\sigma)/\sigma
\)
converges in law to the standard normal distribution: that is, the cumulative
distribution law of the random variable~$Y_{{\sigma}}$ converges to the
integral of the Gaussian density. 
We therefore see that $Y_{\sigma} (x) - Y_{\sigma} (-x)$ tends to $0$ as
${\sigma} \to \infty$. 

In order to deal with asymptotic approximations (in $\sigma$) of $S_{\sigma}$,
we introduce a new continuous function that we will call $b$. We define
 \begin{align*}
  m ({\sigma},x) & := E_{\sigma} - \sigma x\ , \\
  n ({\sigma},x) & := N_{\sigma} (1-\mu_{\sigma}) + \sigma x = N_{\sigma} - m ({\sigma},x) 
  \ .
\end{align*}
Then if ${\sigma}$ is large enough, we define the aforementioned continuous
function $b$ by
\begin{equation}\label{eq:bdef}
  b({\sigma},x) 
  := 
  \frac{\Gamma(N_{\sigma} + 1)}{\Gamma(n({\sigma},x) + 1) \Gamma(m({\sigma},x) + 1)}
  \mu_{\sigma}^{m({\sigma},x)}(1-\mu_{\sigma})^{n({\sigma},x)}
  \ .
\end{equation}
Observe that if $m ({\sigma},x)$ happens to be integral, then 
\[
  b ({\sigma},x) = b_{\sigma} (m ({\sigma},x))
  \ ;
\]
as such, $b ({\sigma},x)$ should be thought of as the collection of continuous
interpolations of these values, with the difference that we have switched to
using the normalized parameter $x$ instead of $m$.

\begin{proposition}\label{pro:lb}
With the previous notation, assume that $E_{\sigma}=
\sigma^2 + 1 + O(1/\sigma^2)$ when $\sigma$ tends to infinity,
and let $\alpha>0$ be any fixed real.
Then for $x \in [-\sigma^{1-\alpha},\sigma^{1-\alpha}]$\ ,
we have
\[
  b(\sigma,x) 
  = 
  \left(\frac{1}{\sqrt{2\pi}}{e^{-\frac{1}{2}x^2}}\right)\frac{1}{\sigma}
  -
  \left(\frac{1}{6\sqrt{2\pi}}x(x^2-3)e^{-\frac{1}{2}x^2}\right)\frac{1}{\sigma^2}
  +
  O\left(\frac{1}{\sigma^3}\right)
  \ .
\]
\end{proposition}
We defer the proof, which is long and technical, to Section~\ref{sec:lem}.

\begin{corollary}\label{cor:binom}
With the notation of Proposition~\ref{pro:lb},
\begin{displaymath}
  b(\sigma,x) - b(\sigma,-x)
  = 
  \left(\frac{1}{3\sqrt{2\pi}} x (3 - x^2) e^{-\frac{1}{2}x^2} \right)
  \frac{1}{\sigma^2} + O \left( \frac{1}{\sigma^3} \right)
  \ .
\end{displaymath}
\end{corollary}

\begin{remark}
  The error terms in the central limit theorem are well-studied, and there are
  powerful tools to deal with them in the literature. However, the usual
  Edgeworth series techniques do not seem to apply here. Indeed, while the
  expressions involved are identical (for example, the polynomial $x(3 - x^2)$
  appearing in Proposition~\ref{pro:lb} is the negative of the third Hermite
  polynomial~\cite[Sec.~3.4]{kolassa}), the interval \(I_{\sigma}\) of
  convergence we obtain, which is optimal, is not the interval predicted by a
  formal application of Edgeworth series we found in the literature.
\end{remark}

\begin{figure}[htbp ]
  \centering
  \includegraphics[height=4cm]{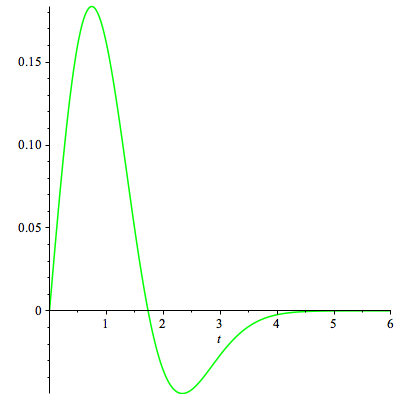}
  \caption{The graph of the approximation $\psi : x \mapsto \frac{1}{3 \sqrt{2 \pi}} x
    (x^2 - 3) e^{-x^2 / 2}$.}
  \label{fig:asym}
\end{figure}

We will use these results  in Sections~\ref{sec:bino} and~\ref{sec:pointcount}
with the following parameters. We assume that the probabilities for a plane
curve over \(\FF_q\) to pass through any one of the $q^2 + q + 1$ points of the
projective plane are described by independent and identically distributed random
variables with probability \(1/q + O(1/q^2)\). More precisely, we will apply the
results for the following values.
 \begin{enumerate}
%
 \item In Corollary~\ref{cor:diffsingular} we will have  
     $\sigma = \sqrt{q - 1/q^2}= \sqrt{q} + O(q^{-5/2})$, 
     $\mu_{\sigma} = 1/q$, 
     and 
     $$E_{\sigma}= (q^2+q+1)/q = \sigma^2+ 1 +O(1/q)=\sigma^2+1+O(1/\sigma^2).$$  
 \item  In Corollary~\ref{cor:diffsmooth} we will have  
     $\sigma= \sqrt{q (1 - \frac{1}{q^2 + q + 1})}$, 
     $\mu_{\sigma}=(q+1)/(q^2+q+1)$,
     and 
     $$E_{\sigma} = q+1 = \sigma^2 + 1 + O(1/\sigma^2).$$
\end{enumerate}
In both situations, we meet the hypotheses of Proposition~\ref{pro:lb}.
Note that we can also apply our results to affine plane curves. Indeed, we pass
through $q^2$ points with probability $\mu_\sigma = 1/q$ so we have $\sigma =
\sqrt{q-1}$ and $ E_\sigma = \sigma ^2+1$.

In the arithmetic setting below, the 
normalization will not be $m(\sigma,x)=E_{\sigma} - \sigma x$ but
$q+1-\sqrt{q} x$. Hence we need to check that we can transfer the perturbation
on $x$ induced by this transformation into the $O(1/\sigma ^3)$. Since the main
term of the expansion in Corollary~\ref{cor:binom} is
$x(x^2-3)e^{-\frac{1}{2}x^2}/{\sigma^2}$, if we replace $x$ by
$(m(\sigma,x)-(q+1))/\sqrt{q} = x+O(1/\sigma)$, then the result still holds.
This is the case in our applications.




\section{Relation with the number of points on plane curves over finite fields.}
\label{sec:pointcount}

The relation between the considerations in the previous section and the
distribution of traces of Frobenius of plane curves over finite fields was first
described by Bucur--David--Feigon--Lal\'in~\cite{bucur}. Let $\FF_q$ be the
finite field with $q$ elements, and let $R = \FF_q [x,y,z]$ be the homogeneous
coordinate ring of the projective plane $\PP^2$ over $\FF_q$. For $f$ in $R$, we
let $C_f \subset \PP^2$ be the plane curve defined by $f = 0$. Intuitively, the
probability that a given point $P$ in $\PP^2(\FF_q)$ lies on $C_f$ should equal
$1/q$, since \emph{a priori} the set of possible values for $f$ at $P$ has $q$
elements (and \(P\) is on \(C\) if and only if \(f(P) = 0\)). Supposing
furthermore that the probabilities are independent as $P$ varies over
\(\PP^2(\FF_q)\), we essentially find ourselves in the situation of
Section~\ref{sec:bino}.

More precisely, let $R_d \subset R$ be the subset of homogeneous polynomials of
degree $d$, and consider curves defined by polynomials in $R_d$. Then the
following seminal result from~\cite{bucur} shows that
the intuition above is accurate as long as $d$ is large enough with respect to
$q$. 

\begin{proposition}[{\cite[Prop. 1.6]{bucur}}]\label{prop:bucur}
  Let $B_1,\ldots,B_{q^2+q+1}$ be i.i.d.\ Bernoulli random variables that assume
  the value $1$ with probability $1/q$. If $d \geq q^2 + q$, then 
  \[
    \frac{\# \{ f \in R_d : \#C_f(\FF_q) = n \} }{\# R_d} =
    \Prob(B_1 + \ldots + B_{q^2 + q + 1} = n) 
    \ .
  \]
\end{proposition}

The original proof of Proposition~\ref{prop:bucur} is based on a general result
of Poonen~\cite{poonen-bertini}. However, we will give another proof by
elementary means, improving the bound on $d$ in the process. 

We choose an enumeration $P_1, \ldots, P_{q^2+q+1}$ of the set $\PP^2 (\FF_q)$,
which we in turn lift to $\FF_q$-rational affine representatives
$\Pt_1,\ldots,\Pt_{q^2 + q + 1}$ in $\AAb^3 (\FF_q)$. 
This allows us to define the linear map
\begin{align*}
    L :  R_d (\FF_q) & \longrightarrow \FF_q^{q^2+q+1}\ , \\
    f & \longmapsto (f(\Pt_1),\ldots, f(\Pt_{q^2+q+1}))\ .
\end{align*}
The forms in the kernel of $L$ define the plane curves that pass through all
rational points of the projective plane; these are called \emph{plane filling
curves}. It is known that the kernel of $L$ is $R_d(\FF_q) \cap J$, where $J$
is the ideal generated by $x^q y - y^q x$, $y^q z - z^q y$ and $z^q x - x^q z$
(see~\cite{tallini1,tallini2} and~\cite[Prop.2.1]{homma-filling}).

\begin{lemma}\label{lemma:freeres}
  For $d \geq 2 q - 1$ the map $L$ is surjective.
\end{lemma}
\begin{proof}
  This proof is due to Masaaki Homma (see also his forthcoming
  article~\cite{homma-tobe}). For any $d \geq 2q-1$, the degree $d$ polynomial
  $x^{d-(2q-2)} (y^{q-1}-x^{q-1}) (z^{q-1}-x^{q-1})$ in $R_d$ takes a non-zero
  value at $(1:0:0)$, and $0$ at every other point in $\PP^2(\FF_q)$. Using the
  transitive action of $\PGL_3(\FF_q)$, we see that we can construct degree $d$
  polynomials having the same properties for any rational point of the plane.
  Therefore the evaluation map $L$ is surjective.
\end{proof}

\begin{proposition}\label{prop:bucurplus}
  Proposition~\ref{prop:bucur} holds for $d \geq 2q - 1$.
\end{proposition}
\begin{proof}
  Lemma~\ref{lemma:freeres} shows that under the given hypothesis on $d$ we can
  always find a homogeneous polynomial with a prescribed zero set of cardinality
  $n$ and given non-zero values on the complement. Hence, the cardinality of the
  set of such  polynomials is the order of the kernel of $L$, and does not
  depend on the prescribed zero set. This implies that the proportion of
  polynomials with a zero set of cardinality $n$ follows a binomial distribution
  $\Prob(B_1 + \ldots + B_{q^2 + q + 1} = n)$.
\end{proof}

For a (possibly singular) plane curve $C_f$, we let $t = q+1-\# C_f
(\FF_q)$, and we define the \emph{normalized trace} $x := t / \sqrt{q}$ in
analogy with Section~\ref{sec:bino}. Note that $x$ is not bounded as $q \to
\infty$. Let
\[
  N_{q,d}(x) 
  := 
  \sqrt{q} \cdot \frac{\# \{ f \in R_d : \# C_f(\FF_q) = q+1-\sqrt{q} x\}}{\# R_d}
  \ ;
\]
we want to approximate the difference 
\[
  V_{q,d}(x) := \sqrt{q} \cdot (N_{q,d}(x) - N_{q,d}(-x))
\]
(note the normalization factor $\sqrt{q}$ which appears once in
$N_{q,d}(x)$ and once in $V_{q,d}$).


\begin{corollary} \label{cor:diffsingular}
  For $d \geq 2q-1$, any $\alpha>0$, and $x \in
  [-(\sqrt{q})^{1-\alpha},(\sqrt{q})^{1-\alpha}]$, we get the following
  approximation of $V_{q,d}(x)$ (in the sense of
  Definition~\ref{def:ofamily}):
  \[
    V_{q,d}(x)
    = 
    \frac{1}{3 \sqrt{2 \pi}} x (3-x^2) \cdot e^{-x^2/2} 
    + 
    O\left(\frac{1}{\sqrt{q}}\right)
    \ .
  \]
\end{corollary}
\begin{proof}
    This follows from Proposition~\ref{prop:bucurplus} and Section~\ref{sec:bino}.
\end{proof}

We now restrict our considerations to nonsingular curves. Let $R_d^{\textrm{ns}}
\subset R_d$ be the subset of homogeneous polynomials corresponding to
nonsingular plane curves. 
Since we have restricted to a further subset defined by the non-vanishing of the
rather complicated discriminant form, the sieving process to get the
distribution is correspondingly more involved, and no elementary method seems to
apply. Instead we use another result in~\cite{bucur}.

\begin{theorem}[{\cite[Th.1.1]{bucur}}]\label{thm:bucur}
  Let $B_1, \ldots, B_{q^2+q+1}$ be i.i.d.~Bernoulli random variables
  taking
  the value $1$ with probability $(q + 1)/(q^2 + q + 1)$. 
  If $0 \leq n \leq q^2+q+1$, then
  \begin{align*}
    \frac{\# \{ f \in R_d^{\textrm{ns}} : \#C_f(\FF_q) = n\}}{\#
    R_d^{\textrm{ns}}} ={} & \Prob(B_1 + \ldots + B_{q^2+q+1} = n) \\
                         & {}\times\left(1 + O\left(q^n \left(d^{-1/3} + (d-1)^2 q^{-\min (\lfloor
      \frac{d}{p} \rfloor +1,\frac{d}{3})} + d q^{-\lfloor \frac{d-1}{p} \rfloor
      -1}\right)\right)\right)
      \ .
  \end{align*} 
\end{theorem}
As before, we let
\[
  N^{\textrm{ns}}_{q,d}(x) 
  :=
  \sqrt{q} \cdot \frac{\# \{ f \in R_d^{\textrm{ns}} : \#C_f(\FF_q) =
  q + 1 - \sqrt{q} x\}}{\# R_d^{\textrm{ns}}}
  \ ;
\]
we want to analyze the difference
\[
  V^{\textrm{ns}}_{q,d}(x) 
  := 
 \sqrt{q} \cdot  (N^{\textrm{ns}}_{q,d}(x) - N^{\textrm{ns}}_{q,d}(-x))
  \ .
\]

Using our Lemma, we see that Theorem~\ref{thm:bucur} implies in particular that
the same asympotic formula from Corollary~\ref{cor:diffsingular} holds for
smooth curves, except that the lower bound on $d$ is now doubly exponential
instead of linear in $q$. More precisely, we have the following.

\begin{corollary}\label{cor:diffsmooth}
  For $d \geq q^{3\,(q^2+q+2)}$,  any $\alpha>0$ and $x \in
  [-(\sqrt{q})^{1-\alpha},(\sqrt{q})^{1-\alpha}]$, we get the following
  approximation of $V^{\textrm{ns}}_{q,d}(x)$  (in the sense of
  Definition~\ref{def:ofamily}):
  \[
    V^{\textrm{ns}}_{q,d}(x) = \frac{1}{3 \sqrt{2 \pi}} x (3-x^2) \cdot
    e^{-x^2 / 2}  + O\left(\frac{1}{\sqrt{q}}\right)\ .
  \]
\end{corollary}

\begin{proof}
  When $d \geq q^{3\,(q^2+q+2)}$ the $O()$ term in Theorem~\ref{thm:bucur} is
  $O(q^{-1/2})$; so applying Theorem~\ref{thm:bucur} and Section~\ref{sec:bino}
  yields the result.
\end{proof}

\section{Experimental results and their limitations}
\label{sec:comp}

We consider the special case $d = 4$. The smooth plane quartic curves $C_f$
defined by $f \in R_4^{\textrm{ns}}(\FF_q)$ are precisely the non-hyperelliptic
curves of genus $3$ over $\FF_q$. Since $d$ is now fixed, our previous results
cannot be applied directly. However, as mentioned in the introduction, we wish
to compare the statistical distributions obtained in this way with the
experimental results obtained in~\cite{LRRS14}.

In order to do so, let us recall the notation of~\cite{LRRS14}. Let
$\nn_{q,3}(t)$ denote the number of $\FF_q$-isomorphism classes of
non-hyperelliptic curves of genus $3$ over $\FF_q$ of trace $t$, weighted by the
order of their $\FF_q$-automorphism group:
\begin{equation}\label{eq:weightedsumdef}
  \nn_{q,3}(t) := \sum_{\left\lbrace \substack{C/\FF_q \, \textrm{n.h.\ genus
  3}\\ \textrm{curve with trace} \; t } \right\rbrace /\simeq} \frac{1}{\#
  \Aut_{\FF_q}(C)}.
\end{equation}
This is the most natural way to define $\nn_{q,3}(t)$
(\cf~\cite{behrend,VdG92}). In order to compare our results for different
values of $q$, we renormalize \(\nn_{q,3}(t)\) to
\[ 
  \nn^{\textrm{KS}}_{q,3} (x) 
  := 
  \frac{\sqrt{q}}{q^6 + 1} \nn_{q,3} (t)
  \ 
  \text{ where }
  t := \lfloor \sqrt{q} x \rfloor
  \text{ for }
  x \in [-6,6] .
\]
This coincides with the normalization of the trace distribution in
Katz--Sarnak~\cite{katz-sarnak} (the factor $q^6+1$ being the number of
$\FF_q$-points of the moduli space of non hyperelliptic genus $3$ curves). The
numerical results in~\cite{LRRS14} on the quantity $\nn^{\textrm{KS}}_{q,3}(x)$
and the difference
\[
  \vv_{q,3}^{\textrm{KS}} (x) :=
  \sqrt{q}\,(\nn^{\textrm{KS}}_{q,3} (x) - \nn^{\textrm{KS}}_{q,3} (-x))
\]
are graphically summarized in Figure~\ref{fig:3}.\\

\begin{figure}[htbp]\label{fig:KSref}
  \centering
  \subcaptionbox%
  {Graph of $\nn^{\textrm{KS}}_{q,3}(x)$ \label{fig:1}}%
  {\includegraphics[height=6cm]{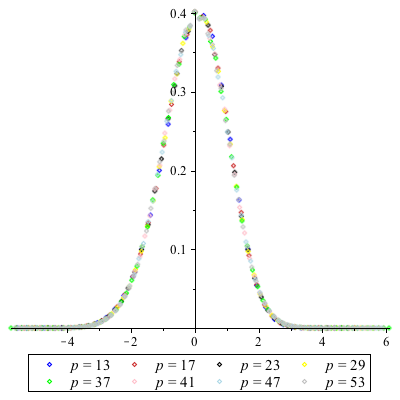}}
  \subcaptionbox%
  {Linear interpolation of $\vv_{q,3}^{\textrm{KS}}(x)$
  \label{fig:2}}%
  {\includegraphics[height=6cm]{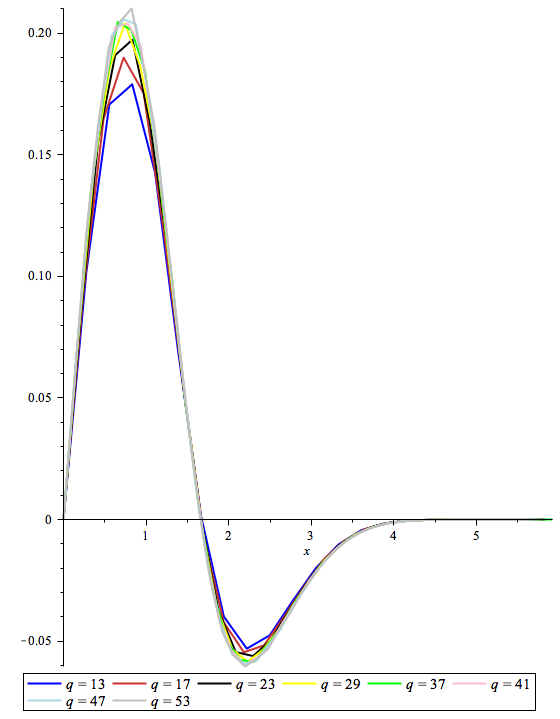}}
  \caption{The numerical trace distributions $\nn^{\textrm{KS}}_{q,3}(x)$ and
    $\vv^{\textrm{KS}}_{q,3}(x)$.}
  \label{fig:3}
\end{figure}

In order to apply our previous results, we need to understand the connection
between~$N_{q,4}^{\textrm{ns}}(x)$ and~$\nn^{\textrm{KS}}_{q,3}(x)$.

\begin{lemma}\label{lem:stacky}
  With $\nn_{q,3}(t)$ defined as above, we have 
  \[
    \# \GL_3(\FF_q)\cdot  \nn_{q,3}(t) 
    = 
    \# \{ f \in R_4^{\textrm{ns}} : \#C_f(\FF_q) = q+1-t\}
    \ .
  \]
\end{lemma}
\begin{proof}
  As smooth plane quartics are isomorphic to their canonical embeddings, an
  $\FF_q$-rational isomorphism between two quartics is induced by an element of
  $\PGL_3 (\FF_q)$. But on the other hand, two ternary forms define the same
  subvariety of $\PP^2$ if and only if they differ by scalar multiplication by
  an element of $\FF_q^{\times}$, so
  \begin{align*}
    \begin{split}
      \# \{ f \in R_4^{\textrm{ns}} : \# C_f(\FF_q) = q + 1 - t \}
      & = \# \FF_q^{\times} \sum_{\{ C / \FF_q \} / \simeq} \frac{\#
        \PGL_3(\FF_q)}{\# \Aut_{\FF_q}(C)} \\
      & = \# \GL_3(\FF_q) \sum_{\{ C / \FF_q\} / \simeq} \frac{1}{\#
        \Aut_{\FF_q}(C)} \\
      & = \# \GL_3(\FF_q) \cdot \nn_{q,3}(t) ,
    \end{split}
  \end{align*}
  where the sums are taken over the same set of curves as in
  Equation~\eqref{eq:weightedsumdef}.
\end{proof}

By Lemma~\ref{lem:stacky},
\[
  \nn^{\textrm{KS}}_{q,3}(x) = \frac{\# R_4^{\textrm{ns}}}{(q^6 + 1) \#
  \GL_3(\FF_q)} N_{q,4}^{\textrm{ns}}(x).
\]
The singular quartic curves in the $14$-dimensional space $\PP R_4$ are
contained in the discriminant locus $D$, a hypersurface of degree $27$. By
\cite[Th.2.1]{sorensen}, we have $\# D(\FF_q) \leq 27 q^{13} +
\frac{q^{13}-1}{q-1}$, so
\[
  \frac{\# R_4^{\textrm{ns}}}{(q^6+1) \# \GL_3(\FF_q)}
  = \frac{\# R_4 - (q-1) \# D(\FF_q)}{(q^6 + 1) (q^3 - 1) (q^3 - q) (q^3 -
    q^2)}
  = 1 + O\left(\frac{1}{q}\right).
\]
Since $\nn^{\textrm{KS}}_{q,3}(x)$ (and therefore  $N_{q,4}^{\textrm{ns}}(x)$)
is uniformly bounded, this implies
\[
  \nn^{\textrm{KS}}_{q,3}(x)
  = N_{q,4}^{\textrm{ns}}(x) \left( 1 + O \left( \frac{1}{q} \right) \right)
  = N_{q,4}^{\textrm{ns}}(x) + O \left( \frac{1}{q} \right)
  \ .
\]

We can now make a graphical comparison of the experimental distribution of
$\vv_{q,3}^{\textrm{KS}}(x)$ for the largest value of $q$ (to wit, $q = 53$)
with the results from Section~\ref{sec:pointcount}. Interpolating the binomial
coefficient as in Equation~\eqref{eq:bdef}, which defined the function $b$ in
Section~\ref{sec:bino}, we define $B_1$, $B_2$, and $B_3$ by
\begin{equation}
  B_i(x) 
  := 
  \sigma_i {N \choose N\mu_i - \sigma_i x} 
  \mu_i^{N\mu_i - \sigma_i x} 
  (1-\mu_i)^{N - (N\mu_i -\sigma_i x)}
  \ ,
\end{equation}
where \(N = q^2 + q + 1\) and
\begin{enumerate}
  \item $\sigma_1 =  \sqrt{q-1/q^2}$ 
    and $\mu_1 = 1/q$
    (corresponding to Corollary~\ref{cor:diffsingular}, \ie, possibly singular
    plane quartics);
  \item $\sigma_2 =  \sqrt{q(1-\frac{1}{q^2+q+1})}$ 
    and $\mu_2 = (q+1)/(q^2+q+1)$ 
    (corresponding to Corollary~\ref{cor:diffsmooth}, \ie, smooth plane
    quartics);
  \item $\sigma_3 =  \sqrt{q}$ and $\mu_3 = 1/q$ (a
    simplified version of these models).
\end{enumerate}

In both cases (1) and (2) we have blithely ignored the constraints on the degree
that were required for the proofs of the corresponding results. Yet as
Figure~\ref{fig:Nvs} shows, the plots of $B_1,B_2,B_3$ and the Gaussian density
function are almost indistinguishable, and all of these functions interpolate
the distribution $\nn^{\textrm{KS}}_{53,3}(x)$ quite well.

\begin{figure}[htbp]
  \centering
  \subcaptionbox%
  {$\nn^{\textrm{KS}}_{53,3}(x)$ against $B_1,B_2,B_3$}  %
  {\includegraphics[height=6cm]{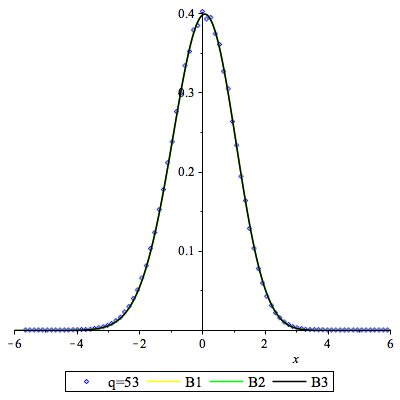}}
  \subcaptionbox%
  {$\nn^{\textrm{KS}}_{53,3}(x)$ against the Gaussian curve} %
  {\includegraphics[height=6cm]{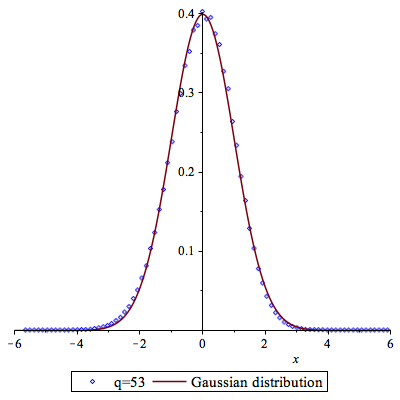}}
  \caption{Comparisons between $\nn^{\textrm{KS}}_{53,3}(x)$ and its
    approximations.}
  \label{fig:Nvs}
\end{figure}

As above, we are led to define
\begin{equation}
  V_i (x) := \sigma\,(B_i (x) - B_i (-x))
  \quad
  \text{for } i = 1,\ 2,\ 3
\end{equation}
and
\begin{equation}
  V^{\textrm{lim}} (x) 
  := 
  \frac{1}{3 \sqrt{2 \pi}} x (3 - x^2) e^{-x^2 / 2} 
  \ .
\end{equation}
This gives rise to the plots in Figure~\ref{fig:Vvs}. 
\begin{figure}[htbp]
  \centering
  \subcaptionbox%
  {$\vv^{\textrm{KS}}_{53,3}(x)$ against $V_1,V_2,V_3$} %
  {\includegraphics[width=6cm]{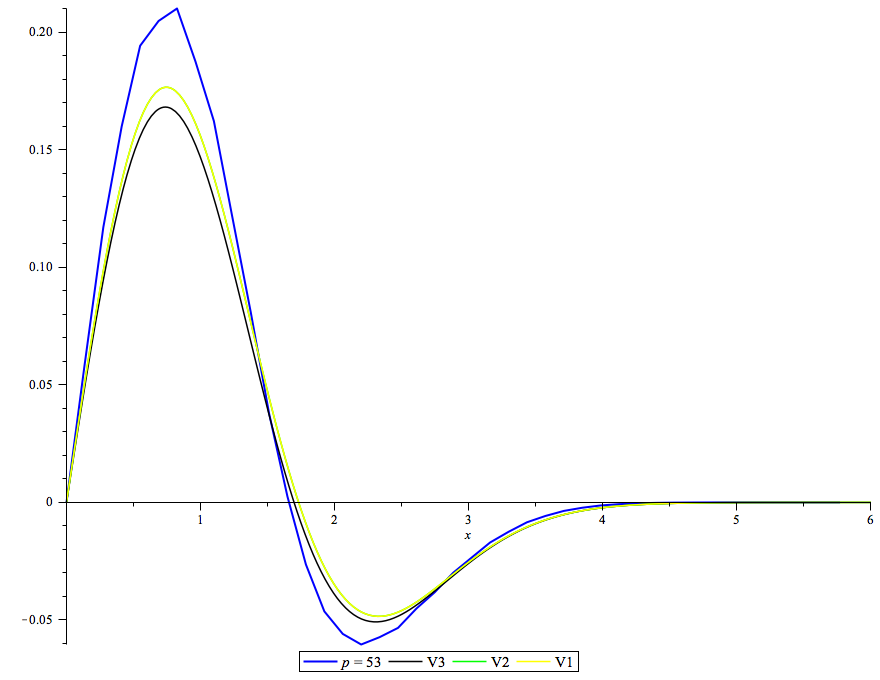}}
  \subcaptionbox%
  {$\vv^{\textrm{KS}}_{53,3}(x)$ against $V_1,V^{\textrm{lim}}$}  %
  {\includegraphics[width=6cm]{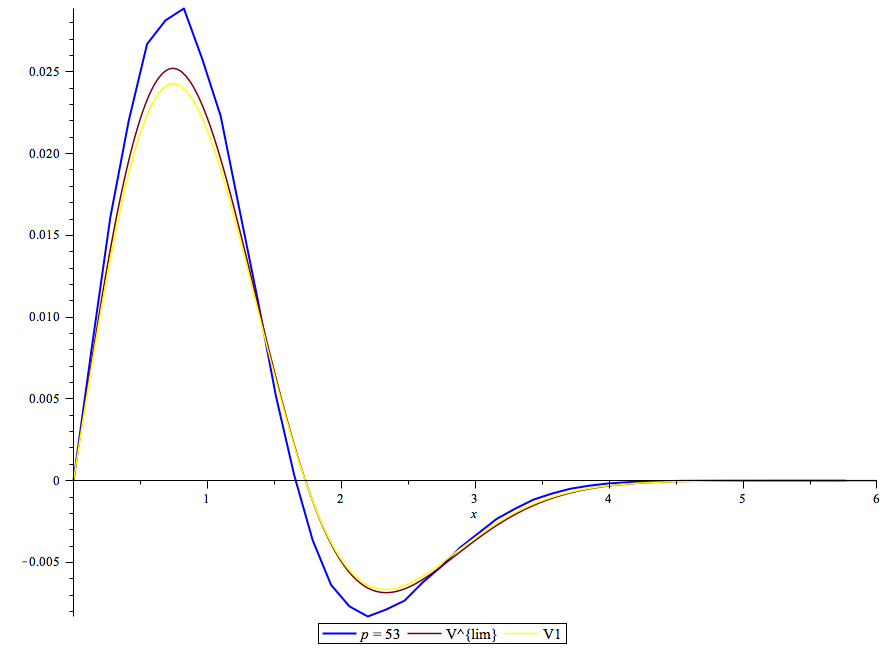}}
  \caption{Comparisons between $\vv^{\textrm{KS}}_{53,3}(x)$ and its
    approximations.}
  \label{fig:Vvs}
\end{figure}
Once more we observe that the various curves closely resemble the distribution
of $\vv^{\textrm{KS}}_{53,3}(x)$. This time around, the agreement is all the
more remarkable in that the limit distribution $V^{\textrm{lim}}$ does not
depend on any parameters, and hence cannot be adjusted.

 \section{Proof of Proposition~\ref{pro:lb}}
\label{sec:lem}

For notational convenience, throughout this proof we write $N$, $E$, $\mu$, $m$,
and $n$ for $N_{\sigma}$, $E_{\sigma}$, $\mu_{\sigma}$, $m({\sigma},x)$ and
$n({\sigma},x)$.

Since $E = \sigma^2 + 1 + O(1/\sigma^2)$ and $x = o(\sigma)$ at worst, we have 
\begin{align*}
    N &= \sigma^4+O(\sigma^2) 
    \ ,
    \\
    \mu &=\frac{1}{\sigma^2}+O\left(\frac{1}{\sigma^4}\right)
    \ ,
    \\
    m &:= E - \sigma x = \sigma^2 - s\,\sigma + 1+O\left(\frac{1}{\sigma^2}\right)
    \ ,
    \\
    n &:= (N-E)+\sigma x = N - m = \sigma^4 + O(\sigma^2)
    \ .
\end{align*}
Stirling's approximation of the Gamma function
\[
\Gamma (z + 1) 
= 
\sqrt{2 \pi z} \left( \frac{z}{e} \right)^z e^{\theta_z}
\quad
\text{with}
\quad 
{|\theta_z|} \leq \frac{1}{12 z}
\]
yields an approximation
\begin{equation}\label{eq:bexp}
  b(\sigma,x) = \frac{1}{\sqrt{2 \pi}} \sqrt{\frac{N}{mn}} \left( \frac{E}{m}
  \right)^m \left( \frac{N -E }{n} \right)^n e^{\theta}  
  \quad
  \text{with}
  \quad
  \theta := \theta_N - \theta_m - \theta_n
  \ .
\end{equation}
Claims 1 through 6 below approximate the individual factors in
Equation~\eqref{eq:bexp}.\bigskip

\noindent
{\bf Claim 1:} 
$\displaystyle\sqrt{\frac{N}{m n}} = \frac{1}{\sigma} \left( 1 + \frac{x}{2
    \sigma} + O \left( \frac{x^2}{\sigma^2} \right) \right)$.
\begin{proof}
  Note that
  \[
    \frac{N}{m n}
    = 
    \frac{1}{\sigma^2} \times \frac{1}{1 + (2\,\mu - 1) \frac{x}{\sigma} - \frac{x^2}{N}}
    = 
    \frac{1}{\sigma^2} \times \frac{1}{1 - (1+O \left( \frac{1}{\sigma^2} \right)) \frac{x}{\sigma} + O \left( \frac{1}{\sigma^2} \right)}
    \ ,
  \]
  and the Taylor expansion of the monotone function $z \mapsto 1/\sqrt{1+z}$ 
  yields Claim 1.
\end{proof}

\noindent
{\bf Claim 2:} 
$\displaystyle e^\theta = 1 + O \left( \frac{1}{\sigma^2} \right)$.
\begin{proof} 
  Estimating $\theta$, we have
  \[
    {|\theta|}
    = 
    {|\theta_N - \theta_m - \theta_n|}
    \leq 
    \frac{1}{12} \left| \frac{1}{N} + \frac{N}{n m} \right|
    \ .
  \]
  From Claim 1 and $N = \sigma^4 + O(\sigma^2)$ we obtain $\theta =
  O(1/\sigma^2)$, which is enough to deduce Claim~2.  
\end{proof}

We will approximate the product of the two middle terms of
Equation~\eqref{eq:bexp} using
\begin{displaymath}
  \left(
    \frac{E}{m}
  \right)^m
  \left(
    \frac{N-E}{m}
  \right)^m
  =
  e^{-(A+B)}
  \quad
  \text{where}
  \quad
  A := - \ln \left( \frac{E}{m} \right)^m 
  \quad
  \text{and}
  \quad
  B := - \ln \left( \frac{N-E}{n} \right)^n 
  \ .
\end{displaymath}
  
\noindent
{\bf Claim 3:} 
$\displaystyle A = -\sigma\,x+\frac{{x}^{2}}{2}+{\frac {{x}^{3}}{6}\,\frac{1}{\sigma}}+O \left( \frac{x^4}{\sigma^{2}} \right)\,.$
\begin{proof}  
  Indeed,
   \begin{align*}
     A = (E - \sigma x) \ln\left(\frac{E - \sigma x}{E}\right) 
     & = E \left(1-\frac{\sigma}{E}\,x\right) \ln\left(1-\frac{\sigma}{E}\,x\right) 
     \\%
     & = 
     E \left(-\frac{\sigma}{E}\,x +
       \frac{1}{2}\left(\frac{\sigma}{E}\,x\right)^2 +
       \frac{1}{6}\left(\frac{\sigma}{E}\,x\right)^3+ O
       \left(\left(\frac{\sigma}{E}\,x\right)^4\right)\right).
  \end{align*}
  The claim follows because
  \( E = \sigma^2 + 1 + O(1/\sigma^2) \)
  and
  \(
    \sigma/E = 1/\sigma - 1/\sigma^3 + O(1/\sigma^5)
  \).
\end{proof}

\noindent
{\bf Claim 4:} $\displaystyle B = \sigma\,x+\frac{{x}^{2}}{2}\,\frac{1}{\sigma^2}+O \left(\frac{x^2}{\sigma^{4}} \right)$.
\begin{proof}
  The proof is analogous to that of Claim 3, using
  \(
    B 
    =
    (N-E) (1+\frac{\sigma}{N-E}\,x )
    \ln(1+\frac{\sigma}{N-E}\,x )
  \)
  and 
  \(
    \sigma/(N-E) = 1/\sigma^3 + O(1/\sigma^5)
  \).
\end{proof}

\noindent
{\bf Claim 5:} 
$\displaystyle A + B = \frac{{x}^{2}}{2} + {\frac {{x}^{3}}{6}\frac{1}{\sigma}} + O \left( \frac{x^4}{\sigma^{2}} \right)$.
\begin{proof}  
  This is immediate from Claims $3$ and $4$.
\end{proof}

\noindent
{\bf Claim 6:} $\displaystyle e^{-(A + B)} = e^{-\frac{x^2}{2}} \left(1 - \frac{x^3}{6 \sigma}\right) + O \left( \frac{1}{\sigma^2} \right)$. 
\begin{proof}
  Taking a first-order Taylor approximation of the monotone
  function $z \mapsto e^z$ shows that 
  \[
    e^z = 1 + z + R (z)  
    \ ,
  \]
  where the remainder term $R$ satisfies $|R (z)| \leq (z^2 / 2) e^{|z|}$. 
  Thus
  \begin{align*}
    \begin{split}
    e^{-(A + B)} & =
    e^{-\frac{x^2}{2}} e^{-\frac{x^3}{6 \sigma} + O \left(
      \frac{x^4}{\sigma^2} \right)} \\
    & = e^{-\frac{x^2}{2}} \left( 1 - \frac{x^3}{6 \sigma} + O \left(
      \frac{x^4}{\sigma^2} \right) \right) + e^{-\frac{x^2}{2}} R \left(
      -\frac{x^3}{6 \sigma} + O \left( \frac{x^ 4}{\sigma^2} \right) \right)
      \ .
    \end{split}
  \end{align*}
  The last term can be estimated by
  \[
    \left| e^{-\frac{x^2}{2}} R \left( -\frac{x^3}{6 \sigma} + O \left(
      \frac{x^4}{\sigma^2} \right) \right) \right|
    \leq \left| \left( \frac{1}{72 \sigma^2} + O \left( \frac{1}{\sigma^4} \right) \right)
      x^6 e^{\left( \left| \frac{x}{6 \sigma} + O \left( \frac{x^2}{\sigma ^2}
      \right) \right| -\frac{1}{2} \right) x^2} \right| 
      \ .
  \]
  Here the assumption $x = o(\sigma)$ plays a role, bounding the product of the
  last factors as $\sigma$ grows, so in fact a stronger estimate holds:
  \[
    e^{-\frac{x^2}{2}} R \left( \frac{x^3}{6 \sigma} + O \left( \frac{x^4}{
      \sigma^2} \right) \right) =
    O \left( \frac{1}{\sigma^2} \right).
  \]
  Since $x^4 e^{-\frac{x^2}{2}}$ is also bounded, regardless of any growth
  assumptions on $x$, we also have
  \begin{equation}\label{eq:expx2kill}
    e^{-\frac{x^2}{2}} \left( O \left( \frac{x^4}{\sigma^2} \right) \right) = O
    \left( \frac{1}{\sigma^2} \right)
    \ ,
  \end{equation}
  which proves the claim.
\end{proof}

Putting all of our estimates together, we obtain a good approximation for $b$
as $\sigma$ tends to infinity.
First, 
\begin{align*}
  \begin{split}
    \sqrt{\frac{N}{m n}} e^{\theta}
    & = \frac{1}{\sigma}\,\left(1 + \frac{x}{2 \sigma} + O \left( \frac{x^2}{\sigma^2} \right)
      \right) \left( 1 + O \left( \frac{1}{\sigma^2} \right) \right) \\
    & = \frac{1}{\sigma} + \frac{x}{2 \sigma^2} + O \left( \frac{x^2}{\sigma^3} \right) 
      \ .
  \end{split}
\end{align*}
Now consider the product
\begin{align}\label{eq:prod}
  \begin{split}
    \sqrt{2 \pi}\, b(\sigma,x) & = \left( \frac{1}{\sigma} + \frac{x}{2 \sigma^2} + O \left( \frac{x^2}{\sigma^3} \right) \right) e^{-(A + B)} \\
    & = e^{-\frac{x^2}{2}} \left( \frac{1}{\sigma} + \frac{x}{2 \sigma^2} + O \left( \frac{x^2}{\sigma^3} \right) \right)  \left( 1 - \frac{x^3}{6 \sigma} +
      O \left( \frac{1}{\sigma^2} \right) \right) 
      \ ;
  \end{split}
\end{align}
its main contribution is given by
\[
  e^{-\frac{x^2}{2}} \left( \frac{1}{\sigma} + \frac{x}{2 \sigma^2} \right) \left( 1 -
    \frac{x^3}{6 \sigma} \right) =  e^{-\frac{x^2}{2}} \left(
    \frac{1}{\sigma}-{\frac {x \left( {x}^{2}-3 \right) }{6\,{\sigma}^{2}}}-{
      \frac {{x}^{4}}{12\,{\sigma}^{3}}}
  \right) 
  \ .
\]
As in the derivation of Equation~\eqref{eq:expx2kill}, the final term in this
sum is $ O \left( {1}/{\sigma^3} \right)$. The same technique shows that the
other cross-terms in the product of Equation~\eqref{eq:prod} are $O \left(
{1}/{\sigma^3} \right)$, which concludes the proof of
Proposition~\ref{pro:lb}.

\bibliographystyle{abbrv}

\bibliography{heuristic-g3}

\end{document}